\documentclass[11pt]{article}

\usepackage[T1]{fontenc}
\usepackage{lmodern}
\usepackage{amsmath,amssymb,amsthm,mathtools}
\usepackage{booktabs,longtable,needspace}
\usepackage[margin=1.15in]{geometry}
\usepackage{microtype}
\usepackage{xurl}
\usepackage[hidelinks]{hyperref}

\hypersetup{
  pdftitle={Type-Voltage Covers and Finite Locally Kneser Graphs},
  pdfauthor={Weiqi Jiang},
  pdfkeywords={Kneser graph, locally homogeneous graph, voltage graph,
    matching complex, simplicial collapse, Burnside lemma}
}

\newtheorem{theorem}{Theorem}[section]
\newtheorem{proposition}[theorem]{Proposition}
\newtheorem{lemma}[theorem]{Lemma}
\newtheorem{corollary}[theorem]{Corollary}
\theoremstyle{definition}
\newtheorem{definition}[theorem]{Definition}
\theoremstyle{remark}

\title{Type-Voltage Covers and Finite Locally Kneser Graphs}
\author{Weiqi Jiang\\Institute of Theoretical Physics, Chinese Academy of Sciences\\jiangweiqi@itp.ac.cn}
\date{}

\begin{document}

\maketitle

\begin{abstract}
For every $d\geq3$ we construct a connected graph of order
$2\binom{3d+1}{d}$ that is locally $K(2d+1,d)$.  Fix a block $A$ and put
$a'=\min(|A|,3d+1-|A|)$.  Among the loopless binary voltages on the fixed
labeled base $K(3d+1,d)$ that depend only on the intersection types with $A$,
the local-neighborhood-preserving assignments form, modulo gauge, an
$\mathbb F_2$-space of dimension $[q^{a'-5}]\binom{d}{3}_q$, with explicit
canonical coordinates.  Adjacent-line
rigidity holds on $2d+2\leq n\leq3d$: every loopless fixed-block
type-invariant local-neighborhood-preserving binary voltage is gauge trivial.
Dropping type
invariance over the fixed labeled $K(10,3)$, we classify all
local-neighborhood-preserving binary voltage assignments and find that
$H^1(M_3(10);\mathbb F_2)$ has dimension 42.  Under
base relabeling by $S_{10}$ these classes form 1,245,395 orbits, of which
1,245,394 consist of connected covers; an explicit class has orbit size 126
and stabilizer $S_5\wr C_2$.  The proofs combine a triangle-voltage criterion and
cohomology with a simplicial collapse and Burnside enumeration.
\end{abstract}



\section{Introduction}

The Kneser graph $K(n,d)$ has the $d$-subsets of an $n$-element set as its
vertices, with adjacency given by disjointness.  A graph is \emph{locally
$L$} if the subgraph induced by every open neighborhood is isomorphic to
$L$.  The standard construction is $K(n+d,d)$, whose neighborhoods are
copies of $K(n,d)$.  Hall proved that it is the only connected example when
$n\geq3d+1$~\cite{hall1987local}.  Brouwer constructed nonstandard examples
on the boundary $n=3d$ and at several sporadic parameter pairs, and raised
the question of finding further finite graphs locally $K(7,3)$ beyond the
standard $K(10,3)$~\cite{brouwer2024some}.

This problem belongs to the wider study of locally homogeneous graphs and
link-preserving coverings.  Nedela developed the covering-space viewpoint
for locally homogeneous graphs and for projections preserving vertex and
edge links~\cite{nedela1993covering,nedela1994covering}; Weetman gave a
general construction theory~\cite{weetman1994construction}.  We use the
voltage-graph language introduced by Gross~\cite{gross1974voltage}.  For a
binary voltage assignment, a base edge records whether its two lifted
fibers are joined straight or crossed.  Preserving the local graph is then
equivalent to requiring zero total voltage on every triangle, so the local
condition becomes a cocycle condition and gauge changes become
coboundaries.

We develop two complementary versions of that idea.  The structured general
branch varies the parameters while fixing a block in the ground set and
restricting to loopless type-invariant voltages; it yields a classification
uniform in $d$, canonical coordinates, and an infinite family.  The exhaustive
fixed-base branch fixes the labeled base $K(10,3)$, allows all binary voltage
assignments, and analyzes the resulting $S_{10}$ symmetry.

We use the convention $[q^m]f(q)=0$ for $m<0$.  The main results are as
follows.

\begin{theorem}\label{thm:v2-main}
Let $d\geq3$.
\begin{enumerate}
\item There is a connected finite graph on $2\binom{3d+1}{d}$ vertices that
  is locally $K(2d+1,d)$.
\item Fix $A\subseteq\Omega$, where $|\Omega|=3d+1$, and put
  $a'=\min(|A|,3d+1-|A|)$.  The local-neighborhood-preserving, loopless,
  fixed-block type-invariant voltages on the fixed labeled $K(3d+1,d)$,
  modulo gauge, form an $\mathbb F_2$-space of dimension
  \[
    [q^{a'-5}]\binom{d}{3}_q.
  \]
  They have canonical coordinates indexed by the triples
  $0\leq k<i<j\leq d-1$ with $k+i+j=a'-2$.
\item If $2d+2\leq n\leq3d$, every loopless fixed-block type-invariant
  local-neighborhood-preserving binary voltage on $K(n+d,d)$ is gauge
  equivalent to zero.
\item There are exactly $2^{42}$ fixed-base gauge classes of binary voltage
  assignments on the fixed labeled base $K(10,3)$ that preserve every local
  neighborhood.  They are parametrized by
  $H^1(M_3(10);\mathbb F_2)$, which has dimension 42; only the zero class is
  disconnected.
\item The induced $S_{10}$ action has 1,245,395 base-relabeling orbits on
  these classes, of which 1,245,394 consist of connected covers.
\item The class of the explicit 240-vertex cover has $S_{10}$-orbit size 126
  and stabilizer $S_5\wr C_2$.
\end{enumerate}
\end{theorem}

The equivalence levels in the theorem are deliberately distinct:
fixed-base gauge equivalence preserves the projection pointwise, a
base-relabeling orbit also permits an automorphism of the base, and neither
relation is asserted to coincide with abstract graph isomorphism after the
projection is forgotten.

Section~\ref{sec:v2-type-complex} begins the type-voltage branch, followed by
the classification in Section~\ref{sec:v2-type-classification} and the
uniform family in Section~\ref{sec:v2-uniform-family};
Sections~\ref{sec:v2-k103-cohomology} and
\ref{sec:v2-k103-symmetry-orbits} give the unrestricted $K(10,3)$ analysis,
and Section~\ref{sec:v2-adjacent-lines} proves the adjacent-line rigidity.

\section{Binary Voltage Covers}

We record the common covering construction used throughout the paper.  All
sheet coordinates and voltage sums in this section are taken in
$\mathbb F_2$.

\begin{definition}\label{def:v2-binary-voltage-lift}
Let $G=(V(G),E(G))$ be a connected simple graph, and let
\[
  \alpha:E(G)\longrightarrow\mathbb F_2
\]
be a voltage assignment on the unoriented edges of $G$.  The associated
binary lift $\widetilde G_\alpha$ is the graph with vertex set
\[
  V(\widetilde G_\alpha)=V(G)\times\mathbb F_2,
\]
in which $(u,\epsilon)$ and $(v,\eta)$ are adjacent exactly when
$uv\in E(G)$ and
\[
  \eta=\epsilon+\alpha(uv).
\]
Its natural projection is
\[
  p:\widetilde G_\alpha\longrightarrow G,
  \qquad p(v,\epsilon)=v.
\]
Thus every vertex has a two-point fiber, and every base edge has exactly two
lifts.
\end{definition}

Relabeling the two sheets independently over each base vertex gives the
standard gauge operation.

\begin{proposition}[Gauge change]\label{prop:v2-gauge-change}
For a function $f:V(G)\to\mathbb F_2$, define the coboundary
\[
  (\delta f)(uv)=f(u)+f(v).
\]
If $\alpha'=\alpha+\delta f$, then
\[
  \Phi_f:\widetilde G_\alpha\longrightarrow\widetilde G_{\alpha'},
  \qquad
  (v,\epsilon)\longmapsto(v,\epsilon+f(v))
\]
is an isomorphism over $G$.  Conversely, a graph isomorphism between two
binary lifts that commutes with projection and preserves each fiber exists
only when their voltage assignments differ by a coboundary.
\end{proposition}

\begin{proof}
Suppose that $\alpha'=\alpha+\delta f$.  The endpoints of the lifted edge
above $uv$ that starts at $(u,\epsilon)$ are
\[
  (u,\epsilon)
  \quad\text{and}\quad
  (v,\epsilon+\alpha(uv)).
\]
The sum of the sheet coordinates of their images under $\Phi_f$ is
\[
  \alpha(uv)+f(u)+f(v)=\alpha'(uv),
\]
so the images are adjacent in $\widetilde G_{\alpha'}$.  The same formula,
together with $\alpha=\alpha'+\delta f$, proves that $\Phi_f$ has an
edge-preserving inverse.  It does not change the base coordinate, and hence
is an isomorphism over $G$.

Conversely, let $\Phi:\widetilde G_\alpha\to\widetilde G_{\alpha'}$ be such
an isomorphism.  Every permutation of a two-element fiber is translation by
a unique element of $\mathbb F_2$.  Thus there is a function
$f:V(G)\to\mathbb F_2$ for which
\[
  \Phi(v,\epsilon)=(v,\epsilon+f(v)).
\]
Applying $\Phi$ to either lifted edge above $uv$ and using the edge rule in
the target gives
\[
  \alpha'(uv)=\alpha(uv)+f(u)+f(v).
\]
Therefore $\alpha'=\alpha+\delta f$.
\end{proof}

The projection is always locally bijective on vertices.  The point of the
next criterion is to determine when it also preserves the edges inside an
open neighborhood.

\begin{proposition}[Triangle criterion]
\label{prop:v2-triangle-criterion}
Fix $x\in V(G)$ and $i\in\mathbb F_2$.  Projection restricts to a bijection
\[
  N_{\widetilde G_\alpha}((x,i))\longrightarrow N_G(x).
\]
This bijection is an isomorphism between the graphs induced on the two open
neighborhoods if and only if every triangle $xuv$ through $x$ satisfies
\begin{equation}\label{eq:v2-triangle-voltage}
  \alpha(xu)+\alpha(uv)+\alpha(vx)=0.
\end{equation}
Consequently, projection induces a graph isomorphism on every open
neighborhood of the lift if and only if every triangle of $G$ has total
voltage zero.
\end{proposition}

\begin{proof}
For every $u\in N_G(x)$, the unique neighbor of $(x,i)$ above $u$ is
\begin{equation}\label{eq:v2-neighbor-lift}
  (u,i+\alpha(xu)).
\end{equation}
This proves the asserted bijection.  Two base neighbors $u,v\in N_G(x)$ are
adjacent precisely when $xuv$ is a triangle.  When $uv\in E(G)$, the two
vertices from \eqref{eq:v2-neighbor-lift} above $u$ and $v$ are adjacent in
the lift exactly when
\[
  \alpha(xu)+\alpha(xv)=\alpha(uv),
\]
which is equivalent to \eqref{eq:v2-triangle-voltage}.  When $uv$ is not a
base edge, no lifted edge can join a vertex over $u$ to a vertex over $v$.
Hence the neighborhood bijection is a graph isomorphism exactly when the
triangle identity holds for every neighborhood edge, equivalently for every
triangle through $x$.  Requiring this at all base vertices is the same as
requiring zero total voltage on every triangle of $G$.
\end{proof}

For a walk $W=(v_0,v_1,\ldots,v_m)$ in $G$, write
\[
  \alpha(W)=\sum_{j=1}^{m}\alpha(v_{j-1}v_j)\in\mathbb F_2.
\]
Because the voltages are attached to unoriented edges, reversing a walk does
not change its voltage.

\begin{lemma}[Connectedness criterion]\label{lem:v2-connected-lift}
For connected $G$, the binary lift $\widetilde G_\alpha$ is connected if and
only if some closed walk in $G$ has total voltage $1$.  Equivalently,
$\widetilde G_\alpha$ is connected if and only if $\alpha$ is not a
coboundary.
\end{lemma}

\begin{proof}
The lift of a walk $W$ beginning at $(v_0,i)$ ends at
$(v_m,i+\alpha(W))$.  Thus a closed walk of voltage $1$, based at a vertex
$r$, lifts to a path from $(r,0)$ to $(r,1)$.  Since $G$ is connected, choose
a path from $r$ to any prescribed base vertex $v$.  Lifting that path from
each of $(r,0)$ and $(r,1)$ reaches the two different vertices above $v$.
It follows that every vertex of the lift is reachable from $(r,0)$, so the
lift is connected.

Conversely, if the lift is connected, there is a path from $(r,0)$ to
$(r,1)$.  Its projection is a closed walk at $r$, and the change in sheet
coordinate along the path is the voltage of that walk.  This change is $1$,
so the projected closed walk has voltage $1$.

It remains to justify the coboundary reformulation.  If $\alpha=\delta f$,
then the voltage of a closed walk is zero because the values of $f$ at
successive endpoints cancel in pairs.  Conversely, suppose every closed
walk has voltage zero.  Fix a root $r$, choose for each $v\in V(G)$ a path
$P_v$ from $r$ to $v$, and set $f(v)=\alpha(P_v)$.  For every edge $uv$, the
walk obtained by following $P_u$, then $uv$, and then the reverse of $P_v$
is closed.  Its zero voltage gives
\[
  f(u)+\alpha(uv)+f(v)=0,
\]
and hence $\alpha(uv)=f(u)+f(v)=(\delta f)(uv)$.  Therefore $\alpha$ is a
coboundary exactly when no closed walk has voltage $1$, completing all the
equivalences.
\end{proof}

\begin{corollary}[Cohomological classification]
\label{cor:v2-general-cohomology-classification}
Let $X(G)$ be the clique complex of a connected simple graph $G$.  The
fixed-base gauge-equivalence classes of binary voltages for which the natural
projection induces a graph isomorphism on every corresponding open
neighborhood are naturally parametrized by
\[
  H^1(X(G);\mathbb F_2).
\]
The lift $\widetilde G_\alpha$ is connected if and only if $[\alpha]\ne0$;
the zero class gives $G\sqcup G$.
\end{corollary}

\begin{proof}
Edge voltages are $1$-cochains on $X(G)$, and
Proposition~\ref{prop:v2-triangle-criterion} identifies the neighborhood
condition with membership in $Z^1(X(G);\mathbb F_2)$.  By
Proposition~\ref{prop:v2-gauge-change}, fixed-base gauge changes add
elements of $B^1(X(G);\mathbb F_2)$, giving
$Z^1(X(G);\mathbb F_2)/B^1(X(G);\mathbb F_2)=H^1(X(G);\mathbb F_2)$.
Lemma~\ref{lem:v2-connected-lift} gives the connectedness assertion, and
the zero voltage has the two sheets $G\sqcup G$.
\end{proof}

\section{The Type Complex}\label{sec:v2-type-complex}

Let $\Omega$ have size $3d+1$, fix $A\subseteq\Omega$ with $|A|=a$, and put
$G=K(3d+1,d)$.  For $S\in V(G)$ define its type by
\[
  t(S)=|S\cap A|.
\]
The base $G$ is connected: for any two $d$-sets $S,T$, the complement of
$S\cup T$ has at least $d+1$ points and therefore contains a $d$-set
adjacent to both.
For compactness, $ij$ denotes the unordered pair $\{i,j\}$.
Let $B$ be a loopless graph whose vertices are the possible types and whose
edges are restricted to type pairs occurring on edges of $G$.  Its associated
type-invariant voltage is
\[
  \alpha_B(ST)=\mathbf 1_{\{t(S),t(T)\}\in E(B)}
\]
for every edge $ST$ of $G$.  The binary lift is the one from
Definition~\ref{def:v2-binary-voltage-lift}.

Here and below, loopless means that $\alpha_B(ST)=0$ whenever
$t(S)=t(T)$.  If $s$ is a function on the type set, type switching replaces
$\alpha_B(ST)$ by
\[
  \alpha_B(ST)+s(t(S))+s(t(T)).
\]
Two type assignments are gauge equivalent when they differ by such a
switching.  This is the restriction of the gauge change in
Proposition~\ref{prop:v2-gauge-change} to potentials that are constant on
types.

Complementing $A$ replaces type $i$ by $d-i$.  We henceforth assume
$0\leq a\leq(3d+1)/2$, and set
\[
  u=\min(d,a),\qquad L=\max(0,a-d-1).
\]

\begin{proposition}[Type feasibility]\label{prop:v2-type-feasibility}
The possible types are $0,\ldots,u$.  A pair of distinct types $i<j$ occurs
on an edge of $G$ if and only if
\begin{equation}\label{eq:v2-type-edge-feasibility}
  L\leq i+j\leq a.
\end{equation}
A triple of distinct types $i<j<k$ occurs on a triangle if and only if
\begin{equation}\label{eq:v2-type-face-feasibility}
  i+j+k\in\{a-1,a\}.
\end{equation}
\end{proposition}

\begin{proof}
A type-$i$ vertex requires $i\leq a$ points in $A$ and $d-i\leq3d+1-a$
points outside it.  Under the normalization on $a$, these conditions give
exactly $0\leq i\leq u$.

Two disjoint $d$-sets of types $i,j$ require $i+j\leq a$ points in $A$ and
$2d-i-j\leq3d+1-a$ points outside it.  These inequalities are equivalent to
\eqref{eq:v2-type-edge-feasibility}, and they are sufficient by choosing
disjoint subsets independently in the two blocks.

Three pairwise disjoint $d$-sets use all but one point of $\Omega$.  Their
total type is therefore $a$ or $a-1$.  Conversely,
\eqref{eq:v2-type-face-feasibility} partitions each of
$A$ and its complement into the three required parts and a residual part of
size at most one.
\end{proof}

Let $Y_{d,a}$ be the two-dimensional type complex with vertices
$0,\ldots,u$, edges given by \eqref{eq:v2-type-edge-feasibility}, and faces
given by \eqref{eq:v2-type-face-feasibility}.  Thus the
auxiliary graph is understood through its feasible-pair restriction
\[
  E(B)\subseteq E(Y_{d,a}).
\]

\begin{proposition}[Triangle reduction]\label{prop:v2-type-triangle-reduction}
The loopless type-invariant voltages whose lifts preserve every open
neighborhood are precisely
\[
  Z^1(Y_{d,a};\mathbb F_2).
\]
Modulo type switching, their gauge-equivalence classes are naturally
identified with
\[
  H^1(Y_{d,a};\mathbb F_2).
\]
\end{proposition}

\begin{proof}
If a base triangle has a repeated type, the two equal off-diagonal voltage
contributions cancel over $\mathbb F_2$, while the loop contribution is zero.
For three distinct types, Proposition~\ref{prop:v2-type-feasibility} says that
the triangle is exactly a face of $Y_{d,a}$.  The common triangle criterion,
Proposition~\ref{prop:v2-triangle-criterion}, therefore says that preservation
of all open neighborhoods is equivalent to the cocycle equation on every
face of the type complex.  Finally, a type switching adds the coboundary of
the type function $s$, so quotienting the cocycles by these switchings gives
the stated first cohomology group.
\end{proof}

In this type context, an admissible voltage means a loopless type-invariant
binary voltage whose lift preserves every open neighborhood.

\section{Type Classification}\label{sec:v2-type-classification}

Define
\begin{equation}\label{eq:v2-type-tree}
T_{d,a}=\begin{cases}
\{0j:1\leq j\leq u\},&L=0,\\
\{0j:L\leq j\leq d\}\cup\{id:1\leq i<L\},&L>0,
\end{cases}
\end{equation}
and
\begin{equation}\label{eq:v2-type-free-edges}
F_{d,a}=\{ij:k=a-2-i-j,\ 0\leq k<i<j\leq d-1\}.
\end{equation}
The first set is a spanning tree of the one-skeleton of $Y_{d,a}$.  If $L=0$,
every nonzero vertex is joined directly to $0$.  If $L>0$, then $u=d$; the
vertices $L,\ldots,d$ are joined to $0$, while $1,\ldots,L-1$ are joined to
$d$.  The displayed graph is connected and has $u$ edges, hence is a tree.

For a face $f=(x,y,z)$, with $x<y<z$, put
\begin{equation}\label{eq:v2-type-pivot}
p(f)=\begin{cases}
yz,&x=0,\\
xy,&x>0\text{ and }x+y+z=a,\\
xy,&x>0,\ x+y+z=a-1,\ z=d,\ x<L,\\
xz,&\text{otherwise}.
\end{cases}
\end{equation}

\begin{lemma}[Pivot matching]\label{lem:v2-type-pivot}
The map $p$ is a bijection from the faces of $Y_{d,a}$ onto
$E(Y_{d,a})\setminus(T_{d,a}\cup F_{d,a})$.
\end{lemma}

\begin{proof}
We describe the inverse.  Let $ij$ be an edge outside
$T_{d,a}\cup F_{d,a}$, with $i<j$.  If $i+j\in\{a-1,a\}$, its inverse face
is $(0,i,j)$.  If $i+j=L$, its inverse face is $(i,j,d)$.

In every remaining case, $L<i+j\leq a-2$ and $i>0$.  Put
$k=a-2-i-j$.  Since $ij\notin F_{d,a}$, one has $k\geq i$.  If $k+1<j$, the
inverse face is $(i,k+1,j)$; otherwise it is $(i,j,k+2)$.  The latter is a
valid face because $i+j>L$ implies $k+2\leq d$.  The four inverse cases are
disjoint, and substitution in \eqref{eq:v2-type-pivot} gives back $ij$.
\end{proof}

\begin{lemma}[Type-complex collapse]\label{lem:v2-type-collapse}
The complex $Y_{d,a}$ collapses by a sequence of elementary collapses to the
graph $T_{d,a}\cup F_{d,a}$.
\end{lemma}

\begin{proof}
Draw the dependency graph on the set of faces, with an arrow $g\to f$ when
$g\neq f$ also contains the pivot $p(f)$.  An edge belongs to at most two
faces, since its third type is determined by a total of $a-1$ or $a$.  Define
\[
  \Phi(x,y,z)=z-y-x.
\]
The value of $\Phi$ strictly increases along every arrow.  For a sum-$a$ face
with pivot $xy$, the other face is $(x,y,z-1)$.  For a sum-$(a-1)$ face with
generic pivot $xz$, it is $(x,y+1,z)$.  When $x=0$ and the pivot is $yz$, the
only possible other face is $(1,y,z)$.  In each case the increase is one.  In
the endpoint case of \eqref{eq:v2-type-pivot}, the other face would require
type $d+1$.

The dependency graph is therefore acyclic.  Remove faces in a topological
order.  When $f$ is removed, every other face containing $p(f)$ is already
gone, so $(f,p(f))$ is an elementary collapse.
Lemma~\ref{lem:v2-type-pivot} shows that the remaining edges are exactly
$T_{d,a}\cup F_{d,a}$.
\end{proof}

For coefficient extraction we use the convention
$[q^m]f(q)=0$ when $m<0$.

\begin{theorem}[Canonical type coordinates]\label{thm:v2-type-classification}
Let $d\geq3$ and $a'=\min(a,3d+1-a)$.  The admissible loopless
type-invariant binary voltage assignments, modulo gauge, have canonical
coordinates indexed by
\[
  0\leq k<i<j\leq d-1,\qquad k+i+j=a'-2.
\]
The quotient space has dimension
\begin{equation}\label{eq:v2-type-dimension}
\dim H^1(Y_{d,a'};\mathbb F_2)
=|F_{d,a'}|
=[q^{a'-5}]\binom{d}{3}_q.
\end{equation}
Moreover, two type-invariant voltages that differ by a coboundary on the full
Kneser graph already differ by a type coboundary.
\end{theorem}

\begin{proof}
Lemma~\ref{lem:v2-type-collapse} reduces the complex to a spanning tree with
the edges $F_{d,a'}$ adjoined.  Switching uniquely kills the tree values and
leaves arbitrary values on those free edges.  These are the asserted
canonical coordinates.

The Gaussian-binomial identity
\[
  \sum_{0\leq k<i<j\leq d-1}q^{k+i+j}=q^3\binom{d}{3}_q
\]
then proves \eqref{eq:v2-type-dimension} by taking the coefficient of
$q^{a'-2}$.

For the last assertion, let $H=S_A\times S_{\Omega\setminus A}$ and suppose an
$H$-invariant voltage is $\delta f$.  For every $h\in H$, connectedness of the
base implies that $f(hS)+f(S)$ is constant in $S$; these constants define a
character of $H$ that is trivial on every vertex stabilizer.  The $H$-orbits
on $V(G)$ are exactly the type classes, since the two symmetric-group factors
act transitively on subsets of each prescribed cardinality.  A type-zero
vertex exists after complement normalization.  Its vertex stabilizer
contains the full $S_A$ factor and an $S_d$ subgroup of the other factor.
Every homomorphism $S_m\to\mathbb F_2$ is trivial or the sign character when
$m\geq2$, and is trivial when $m\leq1$.  Thus neither sign character can
occur: the first is nontrivial on $S_A$ when $a'\geq2$ (and absent when
$a'\leq1$), while the second is nontrivial on $S_d$ because $d\geq3$.  The
character is trivial, so $f$ is $H$-invariant and hence is a type function.
\end{proof}

\begin{corollary}[Number of type classes]\label{cor:v2-type-count}
For every $d\geq3$ and every block size $a$, there are $2^m$ gauge classes,
of which $2^m-1$ yield connected lifts, where
\[
  m=[q^{a'-5}]\binom{d}{3}_q,
  \qquad a'=\min(a,3d+1-a).
\]
In the stable range $d\geq a'\geq5$,
\[
  m=\operatorname{round}\left(\frac{(a'-2)^2}{12}\right).
\]
\end{corollary}

\begin{proof}
By the last assertion of Theorem~\ref{thm:v2-type-classification}, a type
class is nonzero exactly when its voltage is not a coboundary on the full
Kneser graph.  Lemma~\ref{lem:v2-connected-lift} therefore gives the stated
connected-cover count.  When $d\geq a'$, the
upper bound in \eqref{eq:v2-type-free-edges} is inactive, and the free triples
are partitions of
$a'-5$ into at most three parts after subtracting the staircase $(0,1,2)$.
The displayed nearest-integer expression is the standard formula for this
partition number.
\end{proof}

\section{The Uniform Family}\label{sec:v2-uniform-family}

\begin{theorem}[Uniform family]\label{thm:v2-uniform-family}
For every $d\geq3$, there is a connected graph on
$2\binom{3d+1}{d}$ vertices that is locally $K(2d+1,d)$.
\end{theorem}

\begin{proof}
Take $|A|=5$ and let the auxiliary graph have the single edge $12$.  By
Proposition~\ref{prop:v2-type-feasibility}, the types on a triangle sum to
four or five.  Among three nonnegative integers of either sum, the number of
pairs equal to $\{1,2\}$ is even: it is two for $\{1,1,2\}$ and
$\{1,2,2\}$, and zero otherwise.  Every triangle therefore has zero total
voltage, so Proposition~\ref{prop:v2-triangle-criterion} shows that the lift
preserves each local graph.  The open neighborhood of a vertex in the base
$K(3d+1,d)$ is $K(2d+1,d)$.

For connectedness, label $A=\{0,1,2,3,4\}$ and choose disjoint
$(d-3)$-subsets $P,Q$ outside $\{0,1,\ldots,7\}$.  The vertices
\[
\begin{aligned}
S_0&=\{0,1,2\}\cup P,&S_1&=\{3,4,5\}\cup Q,\\
S_2&=\{0,1,6\}\cup P,&S_3&=\{3,5,7\}\cup Q
\end{aligned}
\]
form a closed walk.  Its type sequence is $3,2,2,1,3$, and its voltage
sequence is $0,0,1,0$.  Hence the voltage class is nonzero, and
Lemma~\ref{lem:v2-connected-lift} makes the lift connected.  The base has
$\binom{3d+1}{d}$ vertices, so the lift has the stated order.
\end{proof}

For $a'=5$, the coefficient in
Theorem~\ref{thm:v2-type-classification} is one and the unique free edge is
$12$.  Thus Theorem~\ref{thm:v2-uniform-family} is the first nonzero layer of
the classification, not a separate exceptional construction.

\subsection{Two small cases}

For $d=3$ and $a=5$, the free set is $F_{3,5}=\{12\}$.  The explicit
240-vertex representative constructed later in this paper uses auxiliary
edges $01,13$.  Its symmetric difference with $12$ is $\{01,12,13\}$, the
cut at type 1.  The two representatives are therefore gauge equivalent as
fixed-base covers.  In particular, Theorem~\ref{thm:v2-uniform-family}
recovers that connected graph locally $K(7,3)$.
This lift is $35$-regular and therefore has $4,200$ edges.  Its
$240$-vertex order distinguishes it from the standard $K(10,3)$, which has
$120$ vertices.

Brouwer's $d=5$ example takes $a=8$ and
\[
  E(B)=\{02,05,12,15,24,35\}.
\]
The six face types are
\[
  025,\quad034,\quad035,\quad124,\quad125,\quad134.
\]
Each has even $B$-edge parity.  The canonical tree is
\[
  \{02,03,04,05,15\},
\]
and the free edges are $23,24$.  Switching along the tree gives canonical
coordinate $(1,0)$.  Thus this example represents one of the three nonzero
gauge classes in the two-dimensional space.

\section[Fixed-base Cohomology for K(10,3)]
{Fixed-base Cohomology for $K(10,3)$}
\label{sec:v2-k103-cohomology}

The type-complex theory classifies voltages that depend only on intersection
types with a fixed block.  In this section we drop the type-invariant
restriction and now consider all binary voltages on the
fixed labeled $K(10,3)$.

\subsection{The fixed-base correspondence}

Write $M_3(10)=X(K(10,3))$ for the clique complex of the base, and use
simplicial cochains with coefficients in $\mathbb F_2$.  Thus
$C^0=C^0(M_3(10);\mathbb F_2)$ consists of vertex functions, while
$C^1=C^1(M_3(10);\mathbb F_2)$ consists of functions on the unoriented
edges.  Put
\[
  Z^1=\ker\delta^1,
  \qquad
  B^1=\operatorname{im}\delta^0.
\]
The equation $\delta^1\alpha=0$ says exactly that the voltage sum on the
boundary of each triangle is zero.

\begin{theorem}[Fixed-base classification]
\label{thm:v2-fixed-base-classification}
The gauge-equivalence classes of binary voltages on the fixed labeled base
$K(10,3)$ whose associated projections induce graph isomorphisms on every
induced open neighborhood are naturally in bijection with
\[
  H^1(M_3(10);\mathbb F_2)=Z^1/B^1,
  \qquad
  Z^1=\ker\delta^1,
  \quad
  B^1=\operatorname{im}\delta^0,
\]
where $C^0$ and $C^1$ are the simplicial cochain groups over $\mathbb F_2$.
\end{theorem}

\begin{proof}
Apply Corollary~\ref{cor:v2-general-cohomology-classification} to
$G=K(10,3)$.  Since $X(K(10,3))=M_3(10)$, the resulting quotient is
$Z^1/B^1$ as stated.
\end{proof}

\subsection{Connected classes}

The cohomology class also detects whether the two sheets belong to one
component.

\begin{corollary}[Connected cohomology classes]
\label{cor:v2-connected-cohomology-class}
For a cocycle $\alpha\in Z^1$ on $K(10,3)$, the lift
$\widetilde G_\alpha$ is connected if and only if $[\alpha]\ne0$ in
$H^1(M_3(10);\mathbb F_2)$.  The zero class gives the disjoint union of two
copies of $K(10,3)$.
\end{corollary}

\begin{proof}
This is the connectedness assertion of
Corollary~\ref{cor:v2-general-cohomology-classification}, specialized to
$G=K(10,3)$.
\end{proof}

\subsection[The matching complex M3(10)]{The matching complex $M_3(10)$}

A simplex of $X(K(10,3))$ is a collection of pairwise disjoint three-subsets
of the ten-point ground set.  Consequently
\[
  X(K(10,3))=M_3(10),
\]
the 3-uniform matching complex.  Its numbers of vertices, edges, and
two-simplices are
\begin{align}
  f_0&=120=\binom{10}{3},\label{eq:v2-m310-f0}\\
  f_1&=2100=\frac{120\binom{7}{3}}{2},\label{eq:v2-m310-f1}\\
  f_2&=2800=
    \frac{\binom{10}{3}\binom{7}{3}\binom{4}{3}}{3!}.
    \label{eq:v2-m310-f2}
\end{align}
The second line counts unordered disjoint pairs, and the third counts
unordered pairwise-disjoint triples.  Four disjoint three-subsets would use
twelve points, so there are no higher-dimensional simplices.  For the exact
calculation below, the canonical bases order the triples lexicographically,
then order the disjoint edges and the pairwise-disjoint triangles by the
induced lexicographic order.

Shareshian and Wachs computed the relevant homology over complex coefficients;
their result includes
\[
  \widetilde H_1(M_3(10);\mathbb C)\cong S^{(5,5)},
\]
whose dimension is 42 \cite{shareshianwachs2009top}.  This published
characteristic-zero value is useful context, but it does not determine the
binary rank or dimension over $\mathbb F_2$: a change of coefficients may
detect torsion.  The binary value used here is therefore established by a
separate exact computation.

\begin{proposition}\label{prop:v2-m310-binary-cohomology}
The first binary cohomology of the matching complex satisfies
\[
  \dim_{\mathbb F_2}H^1(M_3(10);\mathbb F_2)=42.
\]
\end{proposition}

\begin{proof}
Since $K(10,3)$ is connected and has 120 vertices, the kernel of
$\delta^0:C^0\to C^1$ consists of the constant functions.  Thus
$\operatorname{rank}\delta^0=119$, without computation.

The coboundary matrices for $\delta^0$ and $\delta^1$ have dimensions
$2100\times120$ and $2800\times2100$, respectively.  Exact row reduction over
$\mathbb F_2$ gives
\[
  \operatorname{rank}\delta^0=119,
  \qquad
  \operatorname{rank}\delta^1=1939.
\]
The supplement reproduces the canonical finite incidence matrices and this
calculation.  The triangle-by-edge matrix for $\delta^1$ is the transpose of
the boundary matrix $\partial_2$, and the two have the same rank.
It follows that
\[
  \dim Z^1=2100-1939=161
\]
and therefore
\[
  \dim H^1=161-119=42.
\]
\end{proof}

\begin{corollary}[Number of fixed-base covers]
\label{cor:v2-fixed-base-cover-count}
Over the fixed labeled base $K(10,3)$ there are exactly
\[
  2^{42}=4{,}398{,}046{,}511{,}104
\]
gauge classes of local-neighborhood-preserving binary lifts.  Only the zero
class is disconnected, so exactly
\[
  2^{42}-1=4{,}398{,}046{,}511{,}103
\]
classes give connected lifts.  Every such binary lift has 240 vertices and is
locally $K(7,3)$.
\end{corollary}

\begin{proof}
Proposition~\ref{prop:v2-m310-binary-cohomology} gives a 42-dimensional
vector space over $\mathbb F_2$, hence $2^{42}$ classes.  Apply
Corollary~\ref{cor:v2-connected-cohomology-class} to remove the zero class.
The vertex and local-graph assertions follow from the definition of a binary
lift and Proposition~\ref{prop:v2-triangle-criterion}.
\end{proof}

\section[Base Automorphisms and Symmetry Orbits for K(10,3)]
{Base Automorphisms and Symmetry Orbits for $K(10,3)$}
\label{sec:v2-k103-symmetry-orbits}

We now allow relabelings of the base graph.  This replaces the individual
elements of the cohomology group by orbits under the full automorphism group
of $K(10,3)$.  The calculation remains relative to the displayed projection
onto the base.

\subsection{The automorphism group of the base}

For $x\in\Omega=\{0,1,\ldots,9\}$, write
\[
  \mathcal S_x=\left\{T\in\binom{\Omega}{3}:x\in T\right\}
\]
for the point star at $x$.

\begin{theorem}[The Kneser base]
\label{thm:v2-aut-kneser}
The point action gives the full automorphism group:
\[
  \operatorname{Aut}(K(10,3))=S_{10}.
\]
\end{theorem}

\begin{proof}
The natural action of $S_{10}$ preserves disjointness of triples and is
therefore an action by graph automorphisms.  It is faithful: a point $x$ is
the intersection of all triples that contain it, so a point permutation that
fixes every vertex of the Kneser graph fixes every point.

An independent set in $K(10,3)$ is the same thing as an intersecting family
of three-subsets.  The Erd\H{o}s--Ko--Rado bound
\cite{erdoskorado1961intersection} gives
\[
  |\mathcal F|\leq \binom{9}{2}=36.
\]
If the total intersection of $\mathcal F$ is empty, the Hilton--Milner bound
\cite{hiltonmilner1967some} sharpens this to
\[
  |\mathcal F|\leq
  \binom{9}{2}-\binom{6}{2}+1=22<36.
\]
Thus a maximum intersecting family has nonempty total intersection.  It is
then contained in a point star, and equality in the 36-element bound forces
it to be the whole star.  Consequently the maximum independent sets are
precisely the ten point stars $\mathcal S_x$.

Every graph automorphism permutes these stars and hence determines a
permutation of $\Omega$.  The resulting homomorphism
\(\operatorname{Aut}(K(10,3))\to S_{10}\) has trivial kernel: if every star
is fixed setwise, then the membership pattern
\(T\in\mathcal S_x\Longleftrightarrow x\in T\) fixes every triple $T$.
Thus the kernel is trivial.  All natural point permutations already occur,
so the homomorphism is onto and the asserted equality follows.  This also
agrees with the general Kneser-graph automorphism result obtained in
\cite{mirafzal2019automorphism}.
\end{proof}

\subsection{The induced cohomology action}

By Theorem~\ref{thm:v2-aut-kneser}, every base automorphism is represented by
a point permutation $g\in S_{10}$, which acts simplicially on the matching
complex.  On an edge cochain the induced action is
\begin{equation}
\label{eq:v2-cochain-action}
  (g\alpha)(\{S,T\})=\alpha(\{g^{-1}S,g^{-1}T\}).
\end{equation}
The action in \eqref{eq:v2-cochain-action} commutes with the coboundary: for
every cochain $c$, \(\delta(gc)=g(\delta c)\).  It follows that the action
preserves \(Z^1=\ker\delta^1\) and
\(B^1=\operatorname{im}\delta^0\), and hence descends to
\[
  H^1(M_3(10);\mathbb F_2)=Z^1/B^1.
\]

For a cocycle $\alpha$, write
$p_\alpha:\widetilde G_\alpha\to K(10,3)$ for the natural projection.

\begin{proposition}[Classification allowing a base relabeling]
\label{prop:v2-base-relabeling-classification}
Let $\alpha,\beta\in Z^1(M_3(10);\mathbb F_2)$ and let
$g\in\operatorname{Aut}(K(10,3))$.  There is an isomorphism
$\Phi:\widetilde G_\alpha\to\widetilde G_\beta$ satisfying
\[
  p_\beta\circ\Phi=g\circ p_\alpha
\]
if and only if $[\beta]=g[\alpha]$ in $H^1(M_3(10);\mathbb F_2)$.
\end{proposition}

\begin{proof}
The natural lift of $g$ is the isomorphism
\[
  \widehat g:\widetilde G_\alpha\longrightarrow\widetilde G_{g\alpha},
  \qquad \widehat g(v,\epsilon)=(gv,\epsilon).
\]
This follows directly from \eqref{eq:v2-cochain-action} and the lift edge
rule.  By Proposition~\ref{prop:v2-gauge-change}, a fixed-base isomorphism
from $\widetilde G_{g\alpha}$ to $\widetilde G_\beta$ exists exactly when
$\beta-g\alpha\in B^1$.  Composing it with $\widehat g$ gives an isomorphism
covering $g$.  Conversely, if $\Phi$ covers $g$, then
$\Phi\circ\widehat g^{-1}$ is a fixed-base isomorphism, so the same
proposition gives $\beta-g\alpha\in B^1$.  This is equivalent to
$[\beta]=g[\alpha]$.
\end{proof}

Thus the orbits below classify the covers up to isomorphisms that cover a
relabeling of $K(10,3)$.

\subsection{Burnside reduction}

The conjugacy classes of $S_{10}$ are indexed by the 42 partitions of 10.
If
\[
  \lambda=1^{m_1}2^{m_2}\cdots10^{m_{10}}\vdash10,
\]
then the size of the associated conjugacy class is
\begin{equation}
\label{eq:v2-conjugacy-class-size}
  |C_\lambda|=\frac{10!}{\prod_i i^{m_i}m_i!}.
\end{equation}
Let $d_\lambda$ denote the dimension of the fixed subspace of $H^1$ for a
representative of cycle type $\lambda$.  Since that subspace has
$2^{d_\lambda}$ elements, combining
\eqref{eq:v2-conjugacy-class-size} with Burnside's lemma gives
\begin{equation}
\label{eq:v2-burnside-orbits}
  N=\frac{1}{10!}\sum_{\lambda\vdash10}|C_\lambda|2^{d_\lambda}.
\end{equation}
Thus it is enough to compute one fixed-space dimension for each partition;
neither all $10!$ permutations nor all $2^{42}$ cohomology classes need be
enumerated.

The following lemma supplies the fixed-space dimensions from edge-coordinate
ranks.  Here the same letter $g$ denotes the induced permutation matrix on
the 2100 edge coordinates.

\begin{lemma}[Fixed-quotient rank]
\label{lem:v2-fixed-quotient-rank}
Let
\[
  Z\in\mathbb{F}_2^{2100\times161},\qquad
  B\in\mathbb{F}_2^{2100\times119}
\]
be full column rank matrices whose column spaces are, respectively, $Z^1$
and $B^1$.  For $g\in S_{10}$, put
\[
  r_g=\operatorname{rank}[(g-I)Z\mid B].
\]
Then
\[
  \dim((H^1)^g)=161-r_g.
\]
\end{lemma}

\begin{proof}
Because $Z$ has full column rank, a cocycle has unique coordinates $z=Zc$.
Its quotient class is fixed by $g$ precisely when $(g-I)Zc\in B^1$.
Consider the linear map
\[
  \begin{aligned}
  \Phi_g:\mathbb{F}_2^{161}\oplus\mathbb{F}_2^{119}
    &\longrightarrow \mathbb{F}_2^{2100},\\
  (c,b)&\longmapsto(g-I)Zc+Bb.
  \end{aligned}
\]
Its rank is $r_g$, so rank--nullity gives
\(\dim\ker\Phi_g=280-r_g\).  A vector $c$ is admissible exactly when some
$b$ makes $(c,b)$ lie in this kernel.  Because $B$ is injective, that $b$ is
unique.  Projection onto the first coordinate is therefore an isomorphism
from $\ker\Phi_g$ onto the admissible $c$-space.

Since $B^1\subseteq Z^1$ and $Z$ identifies $\mathbb F_2^{161}$ with $Z^1$,
the preimage $Z^{-1}(B^1)$ has dimension 119.  By the $g$-invariance of
$B^1$, it is a subspace of the admissible space.  Quotienting by this
subspace gives precisely the fixed part of $Z^1/B^1$.  Hence
\[
  \dim((H^1)^g)=(280-r_g)-119=161-r_g.
\]
\end{proof}

\subsection{The orbit count}

Table~\ref{tab:v2-burnside-data} in
Appendix~\ref{app:v2-burnside-table} records the 42 fixed dimensions obtained
from those ranks by Lemma~\ref{lem:v2-fixed-quotient-rank}, together with
their contributions to \eqref{eq:v2-burnside-orbits}; the supplement
reproduces the finite calculations.  The conjugacy-class sizes sum to
$3,628,800=10!$, while the weighted contributions sum to
$4,519,289,376,000$.  It follows that
\[
  N=\frac{4,519,289,376,000}{3,628,800}=1,245,395.
\]
The zero cohomology class is fixed by every base automorphism and forms one
orbit.  By Corollary~\ref{cor:v2-fixed-base-cover-count}, it is the unique
disconnected fixed-base class.  Removing that orbit leaves exactly
$1,245,394$ nonzero, and hence connected-cover, base-automorphism orbits.

\begin{theorem}[Number of base-automorphism orbits]
\label{thm:v2-base-automorphism-orbits}
The action of $\operatorname{Aut}(K(10,3))=S_{10}$ on fixed-base gauge classes
has $1,245,395$ orbits.  Among them, exactly $1,245,394$ are nonzero; each
corresponds to a connected binary cover.
\end{theorem}

\subsection{The explicit five-subset orbit}

For each $A\in\binom{\Omega}{5}$, put
\[
  t_A(S)=|S\cap A|
\]
for a triple $S$.  On an edge $\{S,T\}$ of $K(10,3)$, define
\begin{equation}
\label{eq:v2-five-subset-voltage}
  \alpha_A(S,T)=1\quad\Longleftrightarrow\quad
  \{t_A(S),t_A(T)\}\in\bigl\{\{0,1\},\{1,3\}\bigr\}.
\end{equation}

\begin{theorem}[The five-subset orbit]
\label{thm:v2-explicit-five-subset-orbit}
The voltages \eqref{eq:v2-five-subset-voltage}, as $A$ ranges over the
five-subsets of $\Omega$, represent exactly 126 fixed-base gauge classes.
These classes form one nonzero $S_{10}$-orbit.  The stabilizer of each class
has order $28{,}800$ and is isomorphic to $S_5\wr C_2$.
\end{theorem}

\begin{proof}
Let $A_0=\{0,1,2,3,4\}$.  For three pairwise-disjoint triples, their
intersection sizes with $A_0$ sum to four or five.  The possible sorted types
are
\[
  \begin{array}{ccc}
  (0,1,3),&(0,2,2),&(1,1,2),\\
  (0,2,3),&(1,1,3),&(1,2,2).
  \end{array}
\]
For the selected type pairs $\{0,1\}$ and $\{1,3\}$, the respective numbers
of selected edges in these six triangles are
\[
  (2,0,0,0,2,0).
\]
Every triangle therefore has even total voltage, so
Proposition~\ref{prop:v2-triangle-criterion} shows directly that
$\alpha_{A_0}$ is a cocycle.  If $g\in S_{10}$, then
\eqref{eq:v2-cochain-action} and \eqref{eq:v2-five-subset-voltage} give
\[
  g\alpha_A=\alpha_{gA}.
\]
Every five-subset is $gA_0$ for some $g$, so every $\alpha_A$ is a cocycle
and the resulting cohomology classes form a single orbit.

We first identify the repetitions in this family.  Write
$A^c=\Omega\setminus A$ and define the zero-cochain
\[
  f_A(S)=\begin{cases}
    1,&t_A(S)\in\{1,2\},\\
    0,&t_A(S)\in\{0,3\}.
  \end{cases}
\]
Since $t_{A^c}(S)=3-t_A(S)$, the possible unordered pairs
$\{t_A(S),t_A(T)\}$ on an edge are
\[
  \{0,1\},\ \{0,2\},\ \{0,3\},\ \{1,1\},\
  \{1,2\},\ \{1,3\},\ \{2,2\},\ \{2,3\}.
\]
Checking these eight pairs in \eqref{eq:v2-five-subset-voltage} gives
\[
  \alpha_A(S,T)+\alpha_{A^c}(S,T)
  =(\delta f_A)(\{S,T\}).
\]
Consequently
\begin{equation}
\label{eq:v2-complement-class}
  [\alpha_A]=[\alpha_{A^c}].
\end{equation}

There are no other repetitions.  By $S_{10}$-equivariance, it is enough to
fix $A=A_0$ and put $r=|A\cap C|$, because the setwise stabilizer of $A$ is
transitive on the five-subsets $C$ with a given value of $r$.  The four
intermediate values of $r$ are separated by the following closed walks.  The
last two columns give their voltages under the indicated cocycles.
\[
\begin{array}{cclcc}
\toprule
r&C&\text{closed walk }W&\alpha_A(W)&\alpha_C(W)\\
\midrule
1&\{0,5,6,7,8\}&\texttt{012--345--016--349--012}&0&1\\
2&\{0,1,5,6,7\}&\texttt{012--345--016--348--012}&0&1\\
3&\{0,1,2,5,6\}&\texttt{012--345--016--357--012}&1&0\\
4&\{0,1,2,3,5\}&\texttt{012--345--016--347--012}&0&1\\
\bottomrule
\end{array}
\]
Each row is checked directly from \eqref{eq:v2-five-subset-voltage}.  A
closed-walk voltage is unchanged by a coboundary, since the vertex values
cancel in pairs.  The unequal entries therefore show that
$[\alpha_A]\ne[\alpha_C]$ when $1\leq r\leq4$.  For $r=5$ the subsets are
equal, and for $r=0$ they are complementary, so
\eqref{eq:v2-complement-class} accounts for every equality.  Hence the 252
five-subsets give exactly
\[
  \binom{10}{5}/2=126
\]
fixed-base gauge classes.  The closed walk
\texttt{012--345--016--357--012} has voltage one for $\alpha_{A_0}$, so this
orbit is nonzero.

Finally, the setwise stabilizer of the unordered partition $\{A,A^c\}$ is
\[
  (S_A\times S_{A^c})\rtimes C_2\cong S_5\wr C_2.
\]
The two direct factors fix the cocycle $\alpha_A$, while a block swap sends it
to $\alpha_{A^c}$ and hence fixes its cohomology class by
\eqref{eq:v2-complement-class}.  Thus this wreath product is contained in the
class stabilizer and has order $2(5!)^2=28{,}800$.  On the other hand,
orbit--stabilizer and the orbit size just proved give
\[
  |\operatorname{Stab}_{S_{10}}([\alpha_A])|=10!/126=28{,}800.
\]
The containment is therefore an equality.
\end{proof}

\section{Adjacent Parameter Lines}\label{sec:v2-adjacent-lines}

We now turn from the constructive line $n=2d+1$ to the adjacent interval.
Fix $d\geq3$, write $n=2d+r$ with $2\leq r\leq d$, put $N=n+d$, let
$\Omega$ have size $N$, and let $G=K(N,d)$.  Fix a block
$A\subseteq\Omega$ with $|A|=a$.  A vertex $S$ of $G$ has type
$t(S)=|S\cap A|$.  Complementing $A$ reverses types by $i\mapsto d-i$, so
we may first take $0\leq a\leq N/2$.

Let $Y_{d,n,a}$ be the two-dimensional complex whose vertices, edges, and
faces are respectively the distinct types occurring on vertices, edges, and
triangles of $G$.

\begin{lemma}[Adjacent-line feasibility]\label{lem:v2-adjacent-feasibility}
A type $i$ occurs precisely when
\[
  \max(0,a-n)\leq i\leq\min(d,a).
\]
A pair of distinct types $i,j$ occurs precisely when both $i$ and $j$ satisfy
the single-type bounds and
\[
  a+d-n\leq i+j\leq a.
\]
A triple of distinct types $i,j,k$ occurs precisely when each of $i,j,k$
satisfies the single-type bounds and
\[
  a+2d-n\leq i+j+k\leq a.
\]
\end{lemma}

\begin{proof}
A type-$i$ $d$-set uses $i$ points in $A$ and $d-i$ outside it.  This gives
the first pair of bounds.  Two disjoint $d$-sets use $i+j$ points in $A$ and
$2d-i-j$ outside it, giving the second bounds; three disjoint sets similarly
use $i+j+k$ and $3d-i-j-k$ points in the two blocks.  Each displayed capacity
condition is sufficient as well: choose the indicated parts successively in
$A$ and independently in its complement.
\end{proof}

Only distinct types are needed in $Y_{d,n,a}$.  In a repeated-type triangle,
the two equal off-diagonal voltage values cancel over $\mathbb F_2$, while a
loop contribution is zero for a loopless voltage.  Thus
Proposition~\ref{prop:v2-triangle-criterion} identifies loopless
type-invariant local-neighborhood-preserving binary voltages with
$Z^1(Y_{d,n,a};\mathbb F_2)$, and type switching adds coboundaries.

\subsection[The base case n=2d+2]{The base case $n=2d+2$}

\begin{proposition}[Base collapse]\label{prop:v2-adjacent-base-collapse}
For every $d\geq3$ and every block size $a$, the complex
$Y_{d,2d+2,a}$ collapses to a tree.  In particular,
\[
  H^1(Y_{d,2d+2,a};\mathbb F_2)=0.
\]
\end{proposition}

\begin{proof}
We first eliminate all cycles.  Here $N=3d+2$.  After complementing the
block if necessary, put
\[
  u=\min(d,a),\qquad L_1=\max(0,a-d-1),\qquad
  L_2=\max(0,a-d-2),
\]
and let $E_s$ be the pairs of types with sum between $L_s$ and $a$.
The faces have sums $a-2$, $a-1$, or $a$.  Define
\[
T=\begin{cases}
\{0j:1\leq j\leq u\},&L_1=0,\\
\{0j:L_1\leq j\leq d\}\cup\{id:1\leq i<L_1\},&L_1>0.
\end{cases}
\]
This is a spanning tree: it is a star in the first case; in the second,
$u=d$, the first group joins $L_1,\ldots,d$ to $0$, and the second joins
$1,\ldots,L_1-1$ to $d$.

First consider the faces inherited from $Y_{d,2d+1,a}$, namely those with
sums $a-1$ and $a$.  Set
\[
  \mathcal F=\{yz:x=a-2-y-z,\ 0\leq x<y<z\leq d-1\}.
\]
For $f=(x,y,z)$ with $x<y<z$, define
\[
q(f)=\begin{cases}
yz,&x=0,\\
xy,&x>0\text{ and }x+y+z=a,\\
xy,&x>0,\ x+y+z=a-1,\ z=d,\ x<L_1,\\
xz,&\text{otherwise}.
\end{cases}
\]
The map $q$ is a bijection from these inherited faces onto
$E_1\setminus(T\cup\mathcal F)$.  Indeed, its displayed edges lie in neither
$T$ nor $\mathcal F$.  Conversely, for $ij$ in that complement, the inverse
is $(0,i,j)$ when $i+j\in\{a-1,a\}$, and $(i,j,d)$ when $i+j=L_1$.
Otherwise $L_1<i+j\leq a-2$.  With $x=a-2-i-j$, the fact that
$ij\notin\mathcal F$ gives $x\geq i$; the inverse is $(i,x+1,j)$ if $x+1<j$
and $(i,j,x+2)$ otherwise.  The latter has $x+2\leq d$, from $i+j>L_1$.
These cases are disjoint and recover $ij$ under $q$.

Extend this matching to every face by
\[
p(x,y,z)=\begin{cases}
xy,&x+y+z=a-2\text{ and }z=d,\\
yz,&x+y+z=a-2\text{ and }z<d,\\
q(x,y,z),&x+y+z\in\{a-1,a\}.
\end{cases}
\]
For the lower endpoint, faces with $z<d$ biject to the edges $\mathcal F$ by
$(x,y,z)\mapsto yz$, while faces with $z=d$ biject to
$E_2\setminus E_1$ by $(x,y,d)\mapsto xy$.  Hence $p$ is a bijection from
the faces of $Y_{d,2d+2,a}$ onto $E_2\setminus T$; both the lower endpoint
$a-2$ and the upper endpoint $a$ are included.

This matching is acyclic.  Draw an arrow $g\to f$ when $g\ne f$ also contains
$p(f)$, and order faces by
\[
  K(x,y,z)=(-x,z,\mathbf 1_{x+y+z=a-1}).
\]
Every arrow satisfies
\[
  K(g)<_{\mathrm{lex}}K(f),
\]
so the key strictly increases from $g$ to $f$.  For a lower-endpoint face
with pivot $yz$, another face increases the least type; with pivot $xy$ and
$z=d$, there is no other face.  For a sum-$a$ pivot $xy$, the largest type
decreases.  For the sum-$(a-1)$ endpoint pivot $xy$, its only possible rival
has third type $d-1$.  If $x=0$ and the pivot is $yz$, the only possible rival
replaces $0$ by $1$.  In the remaining sum-$(a-1)$ case, a pivot $xz$ changes
the middle type by one and changes the last key coordinate from $1$ to $0$.
Thus the dependency graph is acyclic.  Removing faces in a topological order
makes each selected edge free, so elementary collapses remove all of
$E_2\setminus T$.  Therefore $Y_{d,2d+2,a}$ collapses to the tree $T$.
This also covers $a=0$ and the cases in which one or more matching families
are empty.  At the upper normalized endpoint, the only additional case occurs
for even $d$ at $a=(3d+2)/2$; then $L_1=d/2$, and the same map and key argument
apply.  A tree has vanishing first cohomology, proving the final assertion.
\end{proof}

\subsection{Induction across the adjacent interval}

We next keep the block size $a$ fixed and increase $r$.  Normalize $a$ for
the terminal value of $r$, so
$a\leq\lfloor(3d+r)/2\rfloor\leq2d$.  If this $a$ is outside the half-range
for $r=2$, complement it there; type reversal gives an isomorphic complex.
Hence Proposition~\ref{prop:v2-adjacent-base-collapse} applies to every block
size needed for the induction.

Suppose $3\leq s\leq r$ and compare $Y_{d,2d+s,a}$ with
$Y_{d,2d+s-1,a}$.  Their vertex sets agree.  New edges have the sole possible
sum $q=a-d-s$.  If $q\leq0$, no distinct nonnegative types have this sum.  If
$q\geq1$, normalization gives $q\leq d-s<d$, so the endpoints $i,j$ are less
than $d$.  The face $(i,j,d)$ has sum $a-s$ and its other two edges have sums
at least $q+1$, exactly the old lower bound.  Any face on the new edge $ij$
has third type $k$ with
\[
  a-s\leq i+j+k\leq a,
\]
and therefore $d\leq k\leq d+s$; hence $k=d$.  Each new edge is consequently
paired with its unique new minimal-sum face.  After these elementary
collapses, only faces attached along old edges remain.  Such attachments do
not change $\ker(\partial_1)$ and only enlarge
$\operatorname{im}(\partial_2)$.  Induction gives
\[
  \ker(\partial_1)=\operatorname{im}(\partial_2),
  \qquad H_1(Y_{d,n,a};\mathbb F_2)=0.
\]
Since the complexes are finite and the coefficients lie in the field
$\mathbb F_2$, the universal coefficient theorem identifies $H^1$ with the
dual of $H_1$.  Hence
\[
  H^1(Y_{d,n,a};\mathbb F_2)=0.
\]

\begin{theorem}[Adjacent-line type-voltage rigidity]
\label{thm:v2-adjacent-rigidity}
Let $d\geq3$ and $2d+2\leq n\leq3d$, and let $A$ be a fixed block in an
$(n+d)$-point ground set.  Every loopless type-invariant
local-neighborhood-preserving binary voltage on $K(n+d,d)$ relative to this
fixed block is gauge equivalent to zero.
\end{theorem}

\begin{proof}
Write $n=2d+r$.  The preceding vanishing makes every such voltage a type
coboundary, say $\alpha(ST)=f(t(S))+f(t(T))$.  Relabeling a lifted vertex by
$(S,\epsilon)\mapsto(S,\epsilon+f(t(S)))$ makes the voltage zero.  The base
Kneser graph is connected: two $d$-sets have a common disjoint $d$-set since
$n+d\geq3d+2$.  Its zero-voltage lift is therefore the disconnected trivial
double cover.
\end{proof}

The boundary is essential.  At $n=2d+1$, the earlier collapse leaves the
free coordinates that produce the constructive connected covers.  At
$n\geq3d+1$, Hall's recognition theorem applies.

\section{Discussion}

The line $n=2d+1$ supports a constructive family whose order is twice that of
the standard graph.  Within the fixed-block classification, the dimension of
the gauge space grows quadratically with the complement-normalized block size
$a'$ in the stable range $d\geq a'\geq5$.  The adjacent-line calculation shows
a contrasting rigidity: throughout $2d+2\leq n\leq3d$, the same type complex
has no first cohomology.

That contrast has a precise methodological scope.  The rigidity theorem
excludes only loopless fixed-block binary type-voltage constructions; it is
not a nonexistence theorem for locally Kneser graphs.  Likewise, in the
$K(10,3)$ analysis a fixed-base gauge class remembers its projection, base
relabeling permits the natural $S_{10}$ action, and abstract graph
isomorphism may forget the projection altogether.  The 240-vertex example
therefore proves existence but does not determine the least order of a
nonstandard finite graph locally $K(7,3)$.

For arbitrary locally Kneser graphs, the range $2d+2\leq n<3d$ is not settled
here.  Examples are known at
$(n,d)=(8,3)$ and $(12,5)$, and nonstandard connected examples also occur at
$n=3d$~\cite{brouwer2024some}; Hall's theorem leaves only the standard
connected graph for $n\geq3d+1$~\cite{hall1987local}.  Any further examples
in the open range must evade the fixed-block binary type-voltage rigidity,
but the present results impose no restriction on other constructions.

There is also a useful topological comparison.  The sum complexes of Linial,
Meshulam, and Rosenthal select faces by sums in a cyclic group and contain a
full lower skeleton~\cite{linialmeshulamrosenthal2010sum}.  Our type complex
instead has a truncated edge set forced by disjointness, and its boundary
cases control the simplicial collapse.  This makes sum complexes relevant
context without making their homological calculations interchangeable with
the one used here.

\section{Reproducibility}\label{sec:v2-reproducibility}

The combined supplement records both finite-computation branches.  Its
$K(10,3)$ inventory contains the cohomology and orbit calculations together
with the symmetry, extension, and explicit-witness certificates; its general
inventory contains the general-type certificate and the parameter-line
survey.  Generators and independently implemented verifiers are included for
the certificate data.

The article PDF and complete supplement are archived together at
\url{https://doi.org/10.5281/zenodo.22092735}.

After extracting
\texttt{locally-kneser-voltage-covers-supplement-v2.0.0.zip}, the documented
entry points are
\begin{center}
\texttt{make verify-k103},\qquad
\texttt{make verify-general},\qquad
\texttt{make snapshot-check}.
\end{center}
They require GNU make and Python 3.10 or later, with no third-party Python
packages.

The universal statements in this paper are established by the proofs in the
text.  The supplement verifies the finite ranks, orbit sums, explicit
witnesses, symmetry data, general-type dimensions, and bounded
parameter-line calculations used in those proofs.

\section*{Acknowledgments}
I thank Andries E. Brouwer for suggesting the auxiliary-graph formulation,
for pointing out the $d=5$ example, and for posing the question of what lies
beyond the line $n=2d+1$.
\appendix
\section{Burnside Data}
\label{app:v2-burnside-table}

The table supplies the finite fixed-space data used in
Theorem~\ref{thm:v2-base-automorphism-orbits}.
The rows below are indexed by cycle type in decreasing lexicographic
partition order.  The fourth column is the number of fixed cohomology
classes, and the final column is its product with the conjugacy-class
size.  A dash in the total row marks a quantity that is not meaningfully
additive across conjugacy classes.

\begingroup
\small
\setlength{\tabcolsep}{4pt}
\begin{longtable}{@{}lrrrr@{}}
\caption{Burnside contributions for the $S_{10}$-action.}
\label{tab:v2-burnside-data}\\
\toprule
Cycle type & Class size & \shortstack{Fixed\\dimension} & \shortstack{Fixed\\classes} & \shortstack{Weighted\\contribution} \\
\midrule
\endfirsthead
\toprule
Cycle type & Class size & \shortstack{Fixed\\dimension} & \shortstack{Fixed\\classes} & \shortstack{Weighted\\contribution} \\
\midrule
\endhead
$10$ & $362{,}880$ & $6$ & $64$ & $23{,}224{,}320$ \\
$9+1$ & $403{,}200$ & $4$ & $16$ & $6{,}451{,}200$ \\
$8+2$ & $226{,}800$ & $6$ & $64$ & $14{,}515{,}200$ \\
$8+1+1$ & $226{,}800$ & $6$ & $64$ & $14{,}515{,}200$ \\
$7+3$ & $172{,}800$ & $2$ & $4$ & $691{,}200$ \\
$7+2+1$ & $259{,}200$ & $4$ & $16$ & $4{,}147{,}200$ \\
$7+1+1+1$ & $86{,}400$ & $6$ & $64$ & $5{,}529{,}600$ \\
$6+4$ & $151{,}200$ & $6$ & $64$ & $9{,}676{,}800$ \\
$6+3+1$ & $201{,}600$ & $6$ & $64$ & $12{,}902{,}400$ \\
$6+2+2$ & $75{,}600$ & $10$ & $1{,}024$ & $77{,}414{,}400$ \\
$6+2+1+1$ & $151{,}200$ & $8$ & $256$ & $38{,}707{,}200$ \\
$6+1+1+1+1$ & $25{,}200$ & $8$ & $256$ & $6{,}451{,}200$ \\
$5+5$ & $72{,}576$ & $10$ & $1{,}024$ & $74{,}317{,}824$ \\
$5+4+1$ & $181{,}440$ & $2$ & $4$ & $725{,}760$ \\
$5+3+2$ & $120{,}960$ & $2$ & $4$ & $483{,}840$ \\
$5+3+1+1$ & $120{,}960$ & $2$ & $4$ & $483{,}840$ \\
$5+2+2+1$ & $90{,}720$ & $4$ & $16$ & $1{,}451{,}520$ \\
$5+2+1+1+1$ & $60{,}480$ & $4$ & $16$ & $967{,}680$ \\
$5+1+1+1+1+1$ & $6{,}048$ & $6$ & $64$ & $387{,}072$ \\
$4+4+2$ & $56{,}700$ & $12$ & $4{,}096$ & $232{,}243{,}200$ \\
$4+4+1+1$ & $56{,}700$ & $12$ & $4{,}096$ & $232{,}243{,}200$ \\
$4+3+3$ & $50{,}400$ & $6$ & $64$ & $3{,}225{,}600$ \\
$4+3+2+1$ & $151{,}200$ & $4$ & $16$ & $2{,}419{,}200$ \\
$4+3+1+1+1$ & $50{,}400$ & $4$ & $16$ & $806{,}400$ \\
$4+2+2+2$ & $18{,}900$ & $14$ & $16{,}384$ & $309{,}657{,}600$ \\
$4+2+2+1+1$ & $56{,}700$ & $12$ & $4{,}096$ & $232{,}243{,}200$ \\
$4+2+1+1+1+1$ & $18{,}900$ & $12$ & $4{,}096$ & $77{,}414{,}400$ \\
$4+1+1+1+1+1+1$ & $1{,}260$ & $14$ & $16{,}384$ & $20{,}643{,}840$ \\
$3+3+3+1$ & $22{,}400$ & $12$ & $4{,}096$ & $91{,}750{,}400$ \\
$3+3+2+2$ & $25{,}200$ & $10$ & $1{,}024$ & $25{,}804{,}800$ \\
$3+3+2+1+1$ & $50{,}400$ & $10$ & $1{,}024$ & $51{,}609{,}600$ \\
$3+3+1+1+1+1$ & $8{,}400$ & $16$ & $65{,}536$ & $550{,}502{,}400$ \\
$3+2+2+2+1$ & $25{,}200$ & $8$ & $256$ & $6{,}451{,}200$ \\
$3+2+2+1+1+1$ & $25{,}200$ & $8$ & $256$ & $6{,}451{,}200$ \\
$3+2+1+1+1+1+1$ & $5{,}040$ & $10$ & $1{,}024$ & $5{,}160{,}960$ \\
$3+1+1+1+1+1+1+1$ & $240$ & $14$ & $16{,}384$ & $3{,}932{,}160$ \\
$2+2+2+2+2$ & $945$ & $26$ & $67{,}108{,}864$ & $63{,}417{,}876{,}480$ \\
$2+2+2+2+1+1$ & $4{,}725$ & $22$ & $4{,}194{,}304$ & $19{,}818{,}086{,}400$ \\
$2+2+2+1+1+1+1$ & $3{,}150$ & $22$ & $4{,}194{,}304$ & $13{,}212{,}057{,}600$ \\
$2+2+1+1+1+1+1+1$ & $630$ & $24$ & $16{,}777{,}216$ & $10{,}569{,}646{,}080$ \\
$2+1+1+1+1+1+1+1+1$ & $45$ & $28$ & $268{,}435{,}456$ & $12{,}079{,}595{,}520$ \\
$1+1+1+1+1+1+1+1+1+1$ & $1$ & $42$ & $4{,}398{,}046{,}511{,}104$ & $4{,}398{,}046{,}511{,}104$ \\
\midrule
\textbf{Total} & $\mathbf{3{,}628{,}800}$ & -- & -- & $\mathbf{4{,}519{,}289{,}376{,}000}$ \\
\bottomrule
\end{longtable}
\endgroup

\begingroup
\small
\bibliographystyle{abbrv}
\bibliography{references}
\endgroup

\end{document}